\documentclass[10pt]{amsart}
\usepackage{amssymb,amsmath,amsfonts,amsthm}
\usepackage{amsrefs}
\usepackage[hidelinks]{hyperref}
\usepackage{booktabs}
\usepackage{array}
\usepackage{float}
\usepackage{tikz}
\usetikzlibrary{arrows.meta,positioning,automata}

\newtheorem{theorem}{Theorem}[section]
\newtheorem{lemma}[theorem]{Lemma}

\newtheorem{remark}[theorem]{Remark}

\theoremstyle{definition}

\newcommand{\N}{\mathbb{N}}

\newcommand{\rev}{\mathrm{rev}}

\newcommand{\pref}{\mathrm{pref}}

\begin{document}

\title{Geometric Progressions meet Zeckendorf Representations}

\author{Diego Marques}
\address{Departamento de Matem\'atica\\
Universidade de Bras\'ilia\\
Bras\'ilia, DF\\
Brazil}
\email{diego@mat.unb.br}

\author{Pavel Trojovsk\'y}
\address{Department of Mathematics, Faculty of Science, University of Hradec Kr\'alov\'e, Rokitansk\'{e}ho 62, Hradec Kr\'alov\'e 50003, Czech Republic}
\email{pavel.trojovsky@uhk.cz}


\begin{abstract}
Motivated by Erd\H{o}s' ternary conjecture and by recent work of Cui--Ma--Jiang [``Geometric progressions meet Cantor sets'', \textit{Chaos Solitons Fractals} \textbf{163}
(2022), 112567.] on
intersections between geometric progressions and Cantor-like sets in standard
bases, we study the corresponding problem in the Zeckendorf numeration system.
We prove that, for any fixed finite set of forbidden binary patterns, any
integers $u\ge 1$, $q\ge 2$, and any window size $M$, the set
of exponents $n$ for which the Zeckendorf expansion of $u q^n$ avoids the
forbidden patterns within its $M$ least significant digits is either finite or ultimately periodic.

\end{abstract}

\maketitle

\section{Introduction}

The interplay between multiplicative structures (such as geometric progressions)
and digital structures (such as numeration systems with local constraints) is a
fertile source of problems at the interface of number theory, combinatorics,
and symbolic dynamics. A classical example is Erd\H{o}s' ternary conjecture
\cite{Erdos1979}, asserting that for all $n\ge 9$ the base-$3$ expansion of
$2^n$ contains the digit $2$.  Interpreting the middle-third Cantor set as the
set of reals whose ternary expansions avoid the digit $1$, this conjecture can
be viewed as a finiteness statement for the intersection of a geometric
progression with a Cantor-type digit-restricted set.

For standard integer bases $b\ge 2$, Cui, Ma, and Jiang \cite{Cui2022}
developed an algorithmic framework to study intersections of geometric
progressions with Cantor-like subsets of the integers defined by forbidden
blocks in base-$b$ expansions.  Their approach encodes the carry propagation of
multiplication into a finite combinatorial object, leading to effective
decidability results.

In the present paper we pursue an analogous viewpoint in a non-standard
numeration system, namely the \textit{Zeckendorf representation} associated with the
Fibonacci numbers.  Let $(F_n)_{n\ge 0}$ be defined by $F_0=0$, $F_1=1$,
$F_{n+2}=F_{n+1}+F_n$.  By Zeckendorf's theorem, every integer $N\ge 1$ admits a
unique decomposition
\[
   N=\sum_{i=2}^k \varepsilon_i F_i,
\]
with $\varepsilon_i\in\{0,1\}$, $\varepsilon_k=1$, and
$\varepsilon_i\varepsilon_{i+1}=0$ for $2\le i<k$.  We write
$Z(N)=\varepsilon_k\varepsilon_{k-1}\cdots \varepsilon_2$ for this Zeckendorf
word.  Symbolically, the set of Zeckendorf words is the \emph{Golden Mean shift},
i.e., the language over $\{0,1\}$ with the block $11$ forbidden (together with
the usual ``no leading zeros'' convention).

\smallskip

\noindent\emph{Remark on digit order.}  
The representation $Z(N)=\varepsilon_k \varepsilon_{k-1}\cdots \varepsilon_2$ lists digits from the most significant to the least significant.  
For local arithmetic (carry propagation) it is more natural to work with the \emph{reversed} order  
\[
\operatorname{rev}(Z(N))=\varepsilon_2\varepsilon_3\cdots \varepsilon_k,
\]  
i.e., reading from the least significant digit upward.  
All subsequent notation will refer to this LSD-first stream.

\smallskip

Fix a finite family $\mathcal F$ of non-empty binary words, thought of as
additional forbidden patterns.  This defines a Zeckendorf-restricted subset
\[
   \mathcal K_{\mathcal F}
   =\{N\ge 1:\ \rev(Z(N))\ \text{contains no factor from }\mathcal F\},
\]
where $\rev(Z(N))$ lists the Zeckendorf digits from the least significant digit
(the coefficient of $F_2$) upward.  Given integers $u\ge 1$ and $q\ge 2$, we
seek to understand the set of exponents
\[
   S_u=\{n\ge 0:\ uq^n\in \mathcal K_{\mathcal F}\}.
\]
If $\mathcal F=\{11\}$, then $\mathcal K_{\mathcal F}$ coincides with the
full Zeckendorf language and $S_u=\N_{\ge 0}$, so the interesting case is when
$\mathcal F$ imposes genuinely new constraints.

A direct analysis of $S_u$ requires inspecting the entire Zeckendorf expansion
of $uq^n$, whose length grows linearly with $n$.  To obtain an effective result
capturing the local arithmetic, we introduce a window-restricted variant.  Let
$L=\max\{|f|:f\in\mathcal F\}$ and fix $M\ge L$. Since we will take prefixes of least-significant-digit streams, we adopt a padding convention that avoids artificially creating forbidden patterns when the true Zeckendorf expansion is shorter than the window.  Let \(\#\notin\{0,1\}\) be a neutral padding symbol and define the right\-infinite word  
\[
\widetilde{Z}(N):=\operatorname{rev}(Z(N))\,\#^{\omega}\in\{0,1,\#\}^{\mathbb{N}}.
\]  
The symbol \(\#\) never equals \(0\) or \(1\); thus a forbidden binary pattern cannot be created by involving a \(\#\).  
Hence the prefix of length \(M\) of \(\widetilde{Z}(N)\) consists only of genuine Zeckendorf digits whenever \(N\) is large enough, and otherwise contains \(\#\)'s that do not affect the avoidance condition.

We say that a finite word $v\in\{0,1,\#\}^*$ \emph{avoids} $\mathcal F$ if no
$f\in\mathcal F$ occurs as a factor of $v$ (equivalently, occurrences cannot use
the symbol $\#$).  We then define
\[
   \mathcal K_{\mathcal F}^{(M)}
   =\{N\ge 1:\ \pref_M(\widetilde{Z}(N))\ \text{avoids }\mathcal F\},
\]
where for a word \(w = w_0 w_1 w_2 \dots \in \{0,1,\#\}^{\mathbb N}\) and an integer \(M \ge 1\), \(\pref_M(w)\) denotes its prefix of length \(M\), i.e., \(w_0 w_1 \dots w_{M-1}\).  
Thus constraints are enforced only within the first \(M\) least significant Zeckendorf digits. Clearly $\mathcal K_{\mathcal F}\subseteq \mathcal
K_{\mathcal F}^{(M)}$ for every $M$, and with the above padding convention one
has
\[
   \mathcal K_{\mathcal F}=\bigcap_{M\ge L}\mathcal K_{\mathcal F}^{(M)}.
\]
Accordingly, we set
\[
   S_u^{(M)}=\{n\ge 0:\ uq^n\in \mathcal K_{\mathcal F}^{(M)}\}.
\]

Our main result shows that, for each fixed window length $M$, the set
$S_u^{(M)}$ is always ultimately periodic, and the period and preperiod can be
computed effectively.

\begin{theorem}\label{thm:main_window}
Let $u\ge 1$, $q\ge 2$ be integers, and let $\mathcal F$ be a finite family of
non-empty binary words.  Let $L=\max\{|f|:f\in\mathcal F\}$ and fix $M\ge L$.
Then the set
\[
   S_u^{(M)}=\{n\ge 0:\ uq^n\in \mathcal K_{\mathcal F}^{(M)}\}
\]
is either finite or ultimately periodic.  That is, there exist integers $n_0\ge 0$ and $p\ge 1$
such that
\[
   n\in S_u^{(M)}\ \Longleftrightarrow\ n+p\in S_u^{(M)}\qquad(n\ge n_0).
\]
Moreover, one can effectively compute such a pair $(n_0,p)$ and decide whether
$S_u^{(M)}$ is finite or infinite.  In the finite case, the same procedure
outputs all elements of $S_u^{(M)}$.
\end{theorem}

The proof is entirely finite-state.  The crucial input is the fact that
multiplication by a fixed integer $q$ in the Fibonacci numeration system is
realizable by a finite sequential transducer \cite{Frougny1992}.  Reading digits
from the least significant side, this transducer induces, for each $M$, a
deterministic update rule on the set of length-$M$ windows
$\{0,1,\#\}^M$.  After composing with the finite automaton that checks avoidance
of $\mathcal F$ inside the window, we obtain a deterministic dynamical system on
a finite state space.  The orbit associated with the progression $(uq^n)_{n\ge
0}$ is therefore eventually periodic, which yields ultimate periodicity of
$S_u^{(M)}$ as an immediate corollary.  The same finite-state construction gives
an explicit algorithm to compute the eventual period.

Theorem~\ref{thm:main_window} gives a rigorous and completely effective solution
to the window-restricted problem.  While passing from $S_u^{(M)}$ to the global
set $S_u$ entails additional difficulties (one must control where forbidden
patterns may first appear as $n$ varies), the window-restricted variant already
captures the local dynamics of Zeckendorf multiplication and provides a
practical method to detect all exponents for which the low-order Zeckendorf
digits of $uq^n$ satisfy prescribed constraints.

\section{The proof of Theorem \ref{thm:main_window}}

\subsection{Automata-theoretic dictionary} \label{subsec:automata-dictionary}

The proof of our main result relies on concepts from automata theory, which we briefly recall here for readers less familiar with this terminology. Our presentation follows standard references such as Frougny~\cite{Frougny1992} and Allouche \& Shallit \cite{Allouche2003}.

\textbf{Finite automaton.} A \emph{deterministic finite automaton (DFA)} consists of a finite set of \emph{states}, an input alphabet, a transition function that specifies how to move between states when reading a symbol, a designated initial state, and a set of accepting states. Such an automaton processes an input word symbol by symbol, and accepts the word if it ends in an accepting state. DFAs precisely recognize the class of \emph{regular languages}.

\textbf{Transducer.}  
A \emph{finite-state transducer} extends the automaton model by producing output.  
Formally, it is a device with finitely many states that reads an input word over some alphabet and produces a corresponding output word over another alphabet.  
Each transition consumes one input symbol and may produce zero, one, or several output symbols.  
If the transducer always produces exactly one output symbol per input symbol, it is called \emph{synchronous} or \emph{letter-to-letter}.  
In the context of numeration systems, such a transducer is often called a \textit{carry transducer}, because it encodes the propagation of carries during arithmetic operations.  
In this paper, the crucial Lemma~2.1 asserts that multiplication by a fixed integer in the Zeckendorf system can be implemented by a synchronous (carry) transducer.

\textbf{Rational functions.} A function between sets of words is called \emph{rational} if its graph can be recognized by a finite-state transducer. Equivalently, rational functions are exactly those computable by deterministic letter-to-letter transducers, possibly after a bounded initial shift. The theory of rational functions in numeration systems, developed by Frougny~\cite{Frougny1992}, underpins our construction of the multiplication transducer.

\textbf{Ultimate periodicity in finite-state dynamics.} A fundamental observation, used repeatedly in our proof, is that any deterministic process on a finite state space must eventually become periodic. More precisely, if \(X\) is a finite set and \(f: X \to X\) is a function, then for any starting point \(x_0 \in X\) the sequence \(x_{n+1} = f(x_n)\) is \emph{ultimately periodic}: there exist integers \(n_0 \ge 0\) and \(p \ge 1\) such that \(x_{n+p} = x_n\) for all \(n \ge n_0\). This elementary combinatorial fact drives the periodicity conclusion in Theorem~\ref{thm:main_window}.

With these notions at hand, the structure of our argument becomes transparent: we model the least-significant-digit prefixes of \(uq^n\) as states of a finite system, whose update rule is induced by the multiplication transducer. The ultimate periodicity of the window sequence (and hence of the set \(S_u^{(M)}\)) then follows directly from finiteness of the state space.

\subsection{Auxiliary results}

We will use the following automata-theoretic input on Fibonacci (Zeckendorf)
arithmetic.  It is ultimately a consequence of Frougny's general results on
normalization and rational transductions in Pisot numeration systems.

\begin{lemma}\label{prop:frougny}
Let $q\ge 2$ be an integer. There exists a finite-state synchronous
(letter-to-letter) transducer $\mathcal T_q$ over the alphabet $\{0,1,\#\}$,
reading the input from the least significant digit upward, with the following
property: for every $N\ge 1$,
\[
   \mathcal T_q\bigl(\widetilde Z(N)\bigr)=\widetilde Z(qN),
   \qquad \widetilde Z(N):=\rev(Z(N))\,\#^{\omega}.
\]
Moreover, such a transducer can be effectively constructed from $q$.
\end{lemma}

\begin{proof}
We use the framework of rational relations and transducers recalled by Frougny \cite{Frougny1992},
namely that a function is \emph{rational} if and only if its graph is a rational
relation, and that rational relations are closed under composition
\cite[\S2]{Frougny1992}.

\smallskip
\noindent\emph{Rationality of addition and of multiplication by $q$.}
Let $\varphi=(1+\sqrt5)/2$. Since $\varphi$ is a Pisot number, Frougny proves that
normalization in base $\varphi$ is rational on any alphabet and, in particular,
that addition is rational (see Corollary 3.6 and Example 3.4 in \cite{Frougny1992}).
Specializing to the Fibonacci (Zeckendorf) numeration system, this yields that the
graph of the addition map on Zeckendorf representations is a rational relation.
Fix $q\ge2$. Choose any addition chain for $q$ (for instance, the one obtained from
the binary expansion of $q$). Evaluating $N\mapsto qN$ along this chain expresses
multiplication by $q$ as a finite composition of additions and duplications, hence
the graph
\[
   R_q \;:=\;\{(Z(N),Z(qN)):\ N\ge1\}\subset \{0,1\}^*\times\{0,1\}^*
\]
is a rational relation by closure under composition \cite[\S2]{Frougny1992}.

\smallskip
\noindent\emph{Uniform bounded length differences.}
Write $Z(N)=\varepsilon_k\varepsilon_{k-1}\cdots\varepsilon_2$ with $\varepsilon_k=1$.
Then $F_k\le N<F_{k+1}$, hence $\lvert Z(N)\rvert=k-1$.
Let
\[
   C(q):=\min\{c\ge1:\ F_c\ge q\}.
\]

We claim that for every $N\ge1$ with leading index $k$ one has
\begin{equation}\label{eq:length-bound}
   \lvert Z(qN)\rvert \le \lvert Z(N)\rvert + C(q).
\end{equation}
Indeed, from $N<F_{k+1}$ we get $qN<qF_{k+1}\le F_{C(q)}F_{k+1}$. Using the standard
Fibonacci identity
\[
   F_{k+C}=F_C F_{k+1}+F_{C-1}F_k \qquad (k\ge0,\ C\ge1),
\]
we obtain $F_{k+C(q)}\ge F_{C(q)}F_{k+1}\ge qF_{k+1}>qN$. Therefore $qN<F_{k+C(q)}$,
so the leading Fibonacci index of $qN$ is at most $k+C(q)-1$, which implies
$\lvert Z(qN)\rvert\le (k+C(q)-1)-1 < (k-1)+C(q)=\lvert Z(N)\rvert+C(q)$, proving
\eqref{eq:length-bound}. Since $qN\ge N$, one also has $\lvert Z(qN)\rvert\ge \lvert
Z(N)\rvert$, hence
\[
   \bigl|\lvert Z(qN)\rvert-\lvert Z(N)\rvert\bigr|\le C(q)\qquad(N\ge1).
\]
Thus the rational relation $R_q$ has bounded differences in the sense of
\cite[Definition~2.1]{Frougny1992}.

\smallskip
\noindent\emph{Letter-to-letter realization and passage to padded streams.}
By \cite[Proposition~2.1]{Frougny1992}, any rational relation with bounded
differences is realized by a transducer whose edges are labeled by pairs of
letters (i.e.\ a letter-to-letter device), equipped with a bounded initial
function. Applying this to $R_q$ yields a finite letter-to-letter transducer that
recognizes the graph of $N\mapsto qN$ on Zeckendorf words, with an initial
adjustment bounded in length by $C(q)$.

To obtain a synchronous transducer on right-infinite streams, we use the padding
symbol $\#$. Consider the map $w\mapsto \rev(w)\#^\omega$ from $\{0,1\}^*$ into
$\{0,1,\#\}^{\mathbb N}$. The bounded initial adjustment from
\cite[Proposition~2.1]{Frougny1992} can be absorbed into finitely many
$\#$-labeled transitions (since $\#^\omega$ supplies arbitrarily many padding
symbols), after which the transducer proceeds letter-to-letter forever.
This produces a finite synchronous transducer $\mathcal T_q$ on
$\{0,1,\#\}^{\mathbb N}$ such that
\[
   \mathcal T_q(\rev(Z(N))\#^\omega)=\rev(Z(qN))\#^\omega
\]
for every $N\ge1$, i.e.\ $\mathcal T_q(\widetilde Z(N))=\widetilde Z(qN)$.

\smallskip
\noindent\emph{Effectiveness.} The constant \(C(q)\) is explicit because the Fibonacci numbers grow exponentially with ratio \(\varphi\) and so \(C(q)\) can be bounded in terms of \(\log_{\varphi} q\).  
An addition chain for \(q\) is explicit, and the closure operations on transducers (product/construction and composition) are effective in the rational‑relations model~\cite[\S2]{Frougny1992}. Hence \(\mathcal{T}_{q}\) can be constructed effectively from \(q\).
\end{proof}


\begin{figure}[h]
\centering
\begin{tikzpicture}[
  >=Stealth,
  node distance=44mm,
  on grid,
  auto,
  every state/.style={draw, thick, minimum size=11mm, inner sep=1pt},
  lab/.style={font=\small, inner sep=1pt},
  omit/.style={font=\small}
]
  \node[state, initial, initial text={}] (c0) {$c=0$};
  \node[state, right=of c0] (c1) {$c=1$};

  \path[->]
    (c0) edge[loop above] node[lab] {$0/0$} (c0)
    (c1) edge[loop above] node[lab] {$0/1$} (c1);

  \path[->]
    (c0) edge[bend left=16] node[lab, above] {$1/0$} (c1)
    (c1) edge[bend left=16] node[lab, below, yshift=-1mm] {$\#/1$} (c0);

  \path[->]
    (c0) edge[loop below] node[lab] {$\#/\#$} (c0)
    (c1) edge[loop right] node[lab] {$\#/\#$} (c1);

  \node[omit] (dots) [below=18mm of c1] {$\cdots$};
  \path[->, dashed, thick]
    (c1) edge node[lab, right] {$1/\ast$} (dots);

\end{tikzpicture}
\caption{Schematic excerpt of the $q=2$ multiplication transducer (not all
states/transitions shown). It reads $\widetilde Z(N)=\rev(Z(N))\#^\omega$
LSD-first and outputs $\widetilde Z(2N)$; edges are labeled by input/output
pairs. The dashed arrow indicates omitted states.}
\label{fig:transducer}
\end{figure}
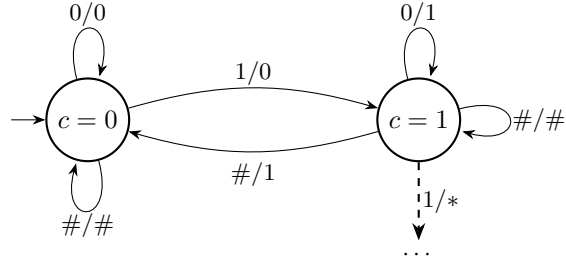

\medskip
The next lemma records the locality principle that will drive the finite-state
dynamics.  Since $\mathcal T_q$ is letter-to-letter, the first $M$ output symbols
only depend on the first $M$ input symbols.

\begin{lemma}\label{lem:locality}
Let $q\ge 2$ and $M\ge 1$. Let $\mathcal T_q$ be a letter-to-letter transducer as in
Lemma~\ref{prop:frougny}. Then there exists a deterministic map
\[
   \Theta_{q,M}:\{0,1,\#\}^M\to \{0,1,\#\}^M
\]
such that for every $N\ge 1$,
\[
   \pref_M(\widetilde{Z}(qN))=\Theta_{q,M}\!\bigl(\pref_M(\widetilde{Z}(N))\bigr).
\]
In other words, \(\Theta_{q,M}\) is the \emph{transition function} of the finite automaton that tracks the length-\(M\) window of the least significant digits under multiplication by \(q\).
\end{lemma}

\begin{proof}
For $v\in\{0,1,\#\}^M$, define
\[
   \Theta_{q,M}(v):=\pref_M\!\bigl(\mathcal T_q(v\#^\omega)\bigr).
\]
Because $\mathcal T_q$ is letter-to-letter, it produces one output symbol per input
symbol, in order. Hence the first $M$ output symbols depend only on the first $M$
input symbols, so $\Theta_{q,M}$ is well-defined.

Now let $v=\pref_M(\widetilde Z(N))$. Since $\widetilde Z(N)$ begins with $v$, the
letter-to-letter property gives
\[
   \pref_M\!\bigl(\mathcal T_q(\widetilde Z(N))\bigr)
   =
   \pref_M\!\bigl(\mathcal T_q(v\#^\omega)\bigr)
   =
   \Theta_{q,M}(v).
\]
Using $\mathcal T_q(\widetilde Z(N))=\widetilde Z(qN)$ from
Lemma~\ref{prop:frougny} yields the claimed identity.
\end{proof}


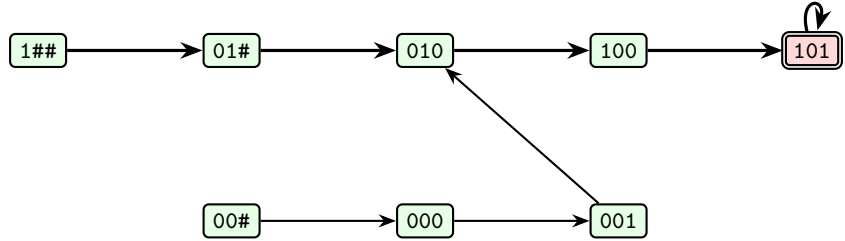
\begin{figure}[h]
\centering
\begin{tikzpicture}[
  >=Stealth,
  node distance=18mm and 18mm,
  every node/.style={font=\small},
  acc/.style={draw, thick, rounded corners=2pt, fill=green!10, inner sep=3.5pt},
  rej/.style={draw, thick, double, rounded corners=2pt, fill=red!15, inner sep=3.5pt},
  orbit/.style={->, very thick},
  edge/.style={->, thick},
]

\node[acc] (w0) at (0,0) {\texttt{1\#\#}};
\node[acc] (w1) [right=of w0] {\texttt{01\#}};
\node[acc] (w2) [right=of w1] {\texttt{010}};
\node[acc] (w3) [right=of w2] {\texttt{100}};
\node[rej] (w4) [right=of w3] {\texttt{101}};

\draw[orbit] (w0) -- (w1);
\draw[orbit] (w1) -- (w2);
\draw[orbit] (w2) -- (w3);
\draw[orbit] (w3) -- (w4);
\draw[orbit] (w4) edge[loop above] (w4);

\node[acc] (a0) [below=of w2] {\texttt{000}};
\node[acc] (a1) [left=of a0] {\texttt{00\#}};
\node[acc] (a2) [right=of a0] {\texttt{001}};

\draw[edge] (a1) -- (a0);
\draw[edge] (a0) -- (a2);
\draw[edge] (a2) -- (w2);

\end{tikzpicture}

\caption{Illustrative excerpt of the window dynamics for $q=2$ and $M=3$.
Vertices are windows $w\in\{0,1,\#\}^3$, coloured green if $w$ avoids $101$ and red
otherwise ($\mathcal F=\{101\}$). Thick arrows indicate a sample forward
trajectory from $w_0=\pref_3(\widetilde Z(1))=\texttt{1\#\#}$.}
\label{fig:theta-window}
\end{figure}

With the auxiliary results in place, we now proceed to the proof of Theorem \ref{thm:main_window}.

\subsection{The proof}

Fix $u\ge 1$, $q\ge 2$, a finite family $\mathcal F$ of non-empty binary words,
and an integer $M\ge 1$. Set $x_n=uq^n$ for $n\ge 0$ and define
\[
   w_n:=\pref_M(\widetilde{Z}(x_n))\in\{0,1,\#\}^M.
\]
By Lemma~\ref{lem:locality}, we have the recursion
\[
   w_{n+1}=\Theta_{q,M}(w_n)\qquad(n\ge 0).
\]
Thus $(w_n)_{n\ge 0}$ is the forward orbit of $w_0$ under a deterministic map on
the finite set $\{0,1,\#\}^M$. Consequently, the sequence $(w_n)$ is ultimately
periodic: there exist $n_0\ge 0$ and $p\ge 1$ such that
\[
   w_{n+p}=w_n\qquad(n\ge n_0).
\]

Now $n\in S_u^{(M)}$ is equivalent, by definition, to the condition that the
length-$M$ word $w_n$ avoids $\mathcal F$. Since the predicate
\[
   P(v):=\text{``$v$ avoids $\mathcal F$''}\qquad(v\in\{0,1,\#\}^M)
\]
depends only on the state $v$, and the sequence $(w_n)$ is ultimately periodic,
the truth values $P(w_n)$ are ultimately periodic as well. Therefore the set
\[
   S_u^{(M)}=\{n\ge 0:\; w_n \text{ avoids } \mathcal F\}
\]
is ultimately periodic.

\medskip
\noindent\emph{Effectiveness.}
Lemma~\ref{prop:frougny} provides an effective construction of the
transducer $\mathcal T_q$, hence of the induced map $\Theta_{q,M}$ in
Lemma~\ref{lem:locality}. Starting from $w_0=\pref_M(\widetilde{Z}(u))$, we
iterate $w_{n+1}=\Theta_{q,M}(w_n)$ and record the first repeated state. The
first repetition yields explicit $(n_0,p)$, and the same computation decides
whether $S_u^{(M)}$ is finite or infinite. In the finite case, the algorithm
outputs all elements of $S_u^{(M)}$ by checking $P(w_n)$ along the finite orbit
segment before the cycle.
\qed

\begin{remark}
The proof avoids any ``infinite regress'' because the evolution of the window
states $w_n$ is governed directly by the finite transducer implementing
$N\mapsto qN$ in LSD-first order. No appeal to Pisot dynamics or decay of
interference from higher-order digits is needed.
\end{remark}

\section{An Illustrative Example}\label{sec:example}

We illustrate the finite-state procedure underlying
Theorem~\ref{thm:main_window} on a concrete instance.

\medskip
\noindent\textbf{Parameters.}
Let $u=1$, $q=2$, and $\mathcal F=\{11, 101\}$.  Since Zeckendorf words already
exclude the factor $11$, we are asking for those exponents $n$ for which the
\emph{first $M$ least significant Zeckendorf digits} of $2^n$ avoid the pattern
$101$.

Recall that if $Z(N)=\varepsilon_k\cdots\varepsilon_2$, then the least
significant digit stream is $\rev(Z(N))=\varepsilon_2\varepsilon_3\cdots
\varepsilon_k$.  The forbidden factor $101$ in $\rev(Z(N))$ corresponds to the
existence of an index $i\ge 4$ such that
\[
   \varepsilon_{i-2}=1,\qquad \varepsilon_{i-1}=0,\qquad \varepsilon_i=1.
\]
Thus, forbidding $101$ means that two $1$'s cannot occur with exactly one $0$
between them in the least-to-most reading; equivalently, within the inspected
window, consecutive $1$'s must be separated by at least two zeros.

\subsection*{The automata and the induced window dynamics}

We fix the window length $M=5$ and work with the padded least-significant-digit
stream
\[
   \widetilde{Z}(N)=\rev(Z(N))\,\#^\omega\in\{0,1,\#\}^{\N}.
\]
Accordingly, the window state we track is
\[
   w_n:=\pref_5\bigl(\widetilde{Z}(2^n)\bigr)\in\{0,1,\#\}^5.
\]
The admissibility condition $n\in S_1^{(5)}$ is exactly that $w_n$ avoids the
binary pattern $101$ (occurrences cannot use the symbol $\#$).

\medskip
\noindent\textbf{Multiplication transducer.}
By Lemma~\ref{prop:frougny} there exists a finite sequential transducer
$\mathcal T_2$ (effectively constructible) which, reading $\widetilde{Z}(N)$
from the least significant side, outputs $\widetilde{Z}(2N)$.  Since the
transducer is synchronous, it induces a deterministic map
\[
   \Theta_{2,5}:\{0,1,\#\}^5\to\{0,1,\#\}^5,\qquad
   w\mapsto \pref_5\bigl(\mathcal T_2(w\#^\omega)\bigr),
\]
and Lemma~\ref{lem:locality} gives
\[
   w_{n+1}=\Theta_{2,5}(w_n)\qquad(n\ge 0).
\]
Thus $(w_n)$ is the forward orbit of $w_0=\pref_5(\widetilde{Z}(1))$ under a
deterministic map on the finite set $\{0,1,\#\}^5$.

\medskip
\noindent\textbf{Pattern detector.}
To decide whether $w_n$ avoids $\{101\}$ one may either check directly inside the
length-$5$ word $w_n$, or equivalently run a small DFA that scans the window and
rejects upon seeing $101$.  Because the window length is fixed and small, the
direct check is simplest and completely effective; we use it in the computation
below.

\subsection*{The first iterates and the finite orbit}

We start at $n=0$. Since $Z(1)=1$, we have $\rev(Z(1))=1$, hence
\[
   \widetilde{Z}(1)=1\#\#\#\#\cdots
   \qquad\text{and}\qquad
   w_0=1\#\#\#\#.
\]
This window clearly avoids $101$, so $0\in S_1^{(5)}$.

Iterating $w_{n+1}=\Theta_{2,5}(w_n)$ produces the window sequence for
$2^n$.  For instance, the first few Zeckendorf expansions are
\[
   Z(1)=1,\qquad Z(2)=10,\qquad Z(4)=101,\qquad Z(8)=10000,
\]
hence the corresponding padded windows (LSD first) are
\[
   w_0=1\#\#\#\#,\quad
   w_1=01\#\#\#,\quad
   w_2=101\#\#,\quad
   w_3=00001.
\]
Thus $w_2$ already contains the forbidden factor $101$, so $2\notin S_1^{(5)}$,
whereas $w_1$ and $w_3$ avoid $101$, so $1,3\in S_1^{(5)}$.

Because $\{0,1,\#\}^5$ is finite, the orbit $(w_n)$ is ultimately periodic.  In
this particular instance, an explicit computation (using the algorithm of
Theorem~\ref{thm:main_window}, i.e., iterating $\Theta_{2,5}$ until the first
repeated window) yields a preperiod of length $29$ and a cycle of length $4$.
Moreover, every state in the cycle is rejecting, hence $S_1^{(5)}$ is finite
and equals the set of admissible indices in the preperiod:
\[
   S_1^{(5)}=\{0,1,3,4,6,8,10,28\}.
\]

\begin{table}[h]
\centering
\caption{Evolution of the window state \(w_n = \operatorname{pref}_{5}(\widetilde{Z}(2^{n}))\)}
\label{tab:window_evolution}
\begin{tabular}{c|c|c}
\(n\) & \(w_n\) & \(n \in S_{1}^{(5)}\)? \\ \hline
0 & \(1\#\#\#\#\) & yes \\
1 & \(01\#\#\#\) & yes \\
2 & \(101\#\#\) & no \\
3 & \(00001\) & yes \\
4 & \(10100\) & yes \\
5 & \(00101\) & no \\
\(\vdots\) & \(\vdots\) & \(\vdots\) \\
28 & \(01010\) & yes \\
29 & \(10101\) & no \\
30 & \(01011\) & no \\
31 & \(10110\) & no \\
32 & \(01101\) & no \\
\end{tabular}
\end{table}

\noindent The full orbit can be generated mechanically, only the first few entries and the cycle are shown above (see also Figure \ref{fig:orbit_structure}).

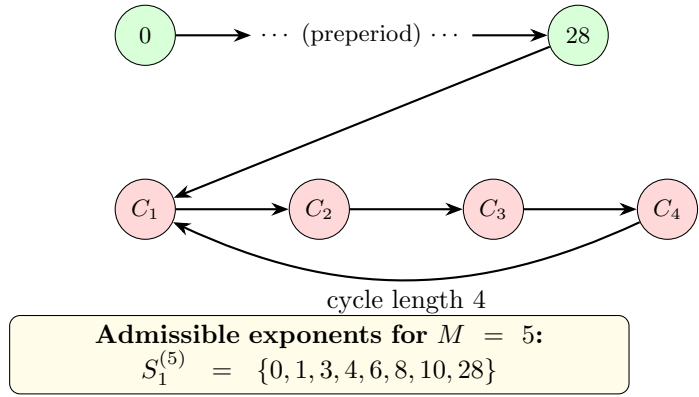
\begin{figure}[h]
\centering
\begin{tikzpicture}[node distance=1.5cm, auto, >=Stealth,
    state/.style={circle, draw, minimum size=0.8cm, font=\small},
    adm/.style={circle, draw, fill=green!15, minimum size=0.8cm, font=\small},
    rej/.style={circle, draw, fill=red!15, minimum size=0.8cm, font=\small},
    pre/.style={rectangle, draw=none, font=\small}]

    \node[adm] (x0) {$0$};
    \node[pre, right=1.0cm of x0] (dots) {$\cdots$ (preperiod) $\cdots$};
    \node[adm, right=1.0cm of dots] (x28) {$28$};

    \node[rej, below=1.5cm of x0] (c1) {$C_1$};
    \node[rej, right=of c1] (c2) {$C_2$};
    \node[rej, right=of c2] (c3) {$C_3$};
    \node[rej, right=of c3] (c4) {$C_4$};

    \draw[->, thick] (x0) -- (dots);
    \draw[->, thick] (dots) -- (x28);
    \draw[->, thick] (x28) -- (c1);

    \draw[->, thick] (c1) -- (c2);
    \draw[->, thick] (c2) -- (c3);
    \draw[->, thick] (c3) -- (c4);
    \draw[->, thick, bend left=25] (c4) to node[midway, below] {cycle length $4$} (c1);

    \node[draw, rectangle, rounded corners, fill=yellow!10, text width=8cm, align=center, below=1.0cm of c2] (info) {
        \textbf{Admissible exponents for $M=5$:}\\
        $S_1^{(5)}=\{0,1,3,4,6,8,10,28\}$
    };

\end{tikzpicture}
\caption{Orbit structure for $(w_n)$ with $u=1$, $q=2$, $\mathcal F=\{101\}$ and
window $M=5$.  The window orbit has a preperiod of length $29$ followed by a
cycle of length $4$.  In this instance the cycle consists entirely of rejecting
states, hence $S_1^{(5)}$ is finite.}
\label{fig:orbit_structure}
\end{figure}

\begin{remark}\label{rem:comparison}
Theorem~\ref{thm:main_window} concerns only the window-restricted sets
$S_u^{(M)}$.  In the present small example one observes that the finite set
$S_1^{(5)}$ coincides with the set of $n$ for which the \emph{entire} Zeckendorf
expansion of $2^n$ avoids $101$.  This coincidence is not guaranteed in general,
but it illustrates that modest window sizes may already capture the full
behaviour for small parameters.
\end{remark}

\section{Concluding remarks}

The main result of this work is intentionally formulated for the
\emph{window-restricted} set
\[
   S_u^{(M)}=\{n\ge 0:\ \pref_M(\widetilde Z(uq^n))\ \text{avoids the forbidden patterns}\},
\]
rather than for the full intersection set
\[
   S_u=\{n\ge 0:\ Z(uq^n)\ \text{globally avoids the forbidden patterns}\}.
\]
This restriction is not merely technical. It isolates the portion of the
problem that can be treated in a completely unconditional and effective way by
finite-state methods: once one has a synchronous transducer for multiplication
by $q$ in the Zeckendorf system, the induced dynamics on length-$M$ prefixes is
a self-map of a finite state space, and eventual periodicity follows from
purely combinatorial considerations. By contrast, passing from prefix
constraints to global constraints requires additional control of the evolution
of the \emph{most significant} digits and of the normalization/carry
propagation beyond any fixed window. These issues are genuinely global and
explain why the full set $S_u$ is more delicate.

\section*{Acknowledgement}
D.M. would like to acknowledge the financial support provided by the National Council for Scientific and Technological Development (CNPq). P.T. was supported by the institutional support for the long-term conceptual development of the research organization, Faculty of Science, University Hradec Králové, No. 2226/2026.

\end{document}